\documentclass[12 pt]{article}
\usepackage{graphicx}
\usepackage{amsmath}
\usepackage{authblk}
\usepackage{amsthm}
\usepackage{amssymb}
\usepackage{xcolor}
\usepackage{tikz}
\usetikzlibrary{positioning}
\usepackage{subcaption}
\usetikzlibrary{graphs}
\usepackage{algorithm}
\usepackage[noend]{algpseudocode}
\usepackage{caption}
\usepackage[colorlinks=true, linkcolor=black, citecolor=black, urlcolor=black]{hyperref}
\usepackage[margin=1in]{geometry}
\usepackage{lineno}

\newtheorem{theorem}{Theorem}[section]

\newtheorem{proposition}[theorem]{Proposition}

\theoremstyle{definition}

\renewcommand\thanks[1]{\protect\footnotetext[0]{#1}}

\theoremstyle{remark}

\begin{document}
\title{A Complete Characterization of all Magic Constants Arising from Distance Magic Graphs}

\author[1]{Ravindra Pawar\thanks{p20200020@goa.bits-pilani.ac.in }}
\author[2]{Tarkeshwar Singh\thanks{tksingh@goa.bits-pilani.ac.in}}
\author[3]{Himadri Mukherjee\thanks{himadrim@goa.bits-pilani.ac.in}}
\author[4]{Jay Bagga\thanks{jbagga@bsu.edu}}
\affil[1,2,3]{Department of Mathematics BITS Pilani KK Birla Goa Campus, Goa, India}
\affil[4]{Department of Computer Science, Ball State University, Indiana, USA.}

\date{}
\maketitle
\noindent
\begin{abstract}

A positive integer $k$ is called a magic constant if there is a graph $G$ along with a bijective function $f$ from $V(G)$ to first $|V(G)|$ natural numbers such that the weight of the vertex $w(v) = \sum_{uv \in E}f(v) =k$  for all $v \in V$. It is known that all odd positive integers greater equal $3$ and the integer powers of $2$, $2^{t}$, $t \ge 6$ are magic constants. In this paper we characterise all positive integers which are magic constants.  
\end{abstract}
\section{Introduction}
Throughout this article, we will assume that $G = (V,E)$ denotes the finite simple graph with $V$ denoting the set of vertices and $E$ the set of edges. For a vertex, $u$, its neighborhood is given by $N(u) = \{v \in V : uv \in E\}$. A positive integer $k$ is called a \textit{magic constant} if there is a graph $G$ along with a bijective function $f: V(G) \rightarrow \{1,2, \dots, |V(G)|\}$ such that the weight of the vertex $w(v) = \sum_{uv \in E}f(v)$,  remains constant and equal to $k$ for all $v \in V$. Such labeling $f$ is called \textit{distance magic labeling}. A graph equipped with distance magic labeling is known as a \textit{ distance magic graph} (for more details, see \cite{dm_survey_arumugam}, \cite{dm_survey_gallian}).  We refer to West \cite{west} for graph-theoretic terminology and notation not covered here.

\smallskip 
Regardless of how the vertices are labeled, the magic constant remains invariant, as previously established  \cite{uniquek, o_uniquek} i.e., it is independent of the distance magic labeling. At the International Workshop on Graph Labeling (IWOGL-2010), Arumugam \cite{dm_survey_arumugam} raised the question to {\it characterise the set of positive integers, which are magic constants}. Since then, this problem has captured the attention of researchers. The integers $1,2,4,6,8,12$ do not  belong to the set of magic constants. On the other hand, all odd integers $\ge 3$  and all integers of the form $2^t~(t \ge 6)$  are confirmed members of the set of magic constants (see, \cite{kamatchit}, \cite{dm_4reg_petr}). A previous study in \cite{dm_algo_fuad} identified a graph with $8$ vertices that possesses a magic constant of $24$. However, characterising the set of integers, which are magic constants, is an ongoing problem in distance magic graphs. We have shown that positive integers of the form $4t + 2$ for $t \ge 3$ and of the form $4t+4$ for $t \ge 8$ are magic constants. Hence, the remaining numbers are $16, 20, 28$, and $32$.

\smallskip 
Bertault et al. \cite{algo_bertault} introduced a heuristic algorithm to identify various classes of labelings. Building on this, Fuad et al. \cite{dm_algo_fuad} refined the algorithm, achieving the construction of all distance magic graphs up to isomorphism for orders up to $9$. This is the well-known algorithm available for constructing distance magic graphs. However, from a computational point of view, this algorithm has limitations. It relies on a collection of all non-isomorphic graphs as input, a formidable task that poses challenges, particularly for higher-order graphs. Generating all non-isomorphic graphs of substantial sizes remains computationally demanding, thereby affecting the practicality of the algorithm.

\smallskip 
In this paper, we completely solve the characterization problem of magic constants. We prove this result by providing constructions of distance magic graphs to obtain specific magic constants.

\smallskip 
We have devised an algorithm to check if there is a distance magic graph for a given number of vertices $n$ and a specific magic constant $k$. We have implemented this algorithm to construct magic graphs for $k = 20, 28$, and $32$. Our search for $k = 16$ shows that no graph admits the magic constant $16$. Using this algorithm, we have generated all distance magic graphs up to isomorphism of order up to $12$. This enhances the existing collection of distance magic graphs \cite{dm_algo_fuad}.

\smallskip
This algorithm serves as a practical tool for researchers and enthusiasts in the field of distance magic graphs. Addressing the general challenge of characterizing all graphs possessing distance magic labelings using algorithms presents a computational hurdle. This is mainly due to the impracticality of generating all non-isomorphic graphs since this becomes increasingly unfeasible as the graph order grows. Consequently, the algorithms devised for characterisation problems often encounter limitations or become computationally intensive when dealing with graphs of larger vertices.
 
 \smallskip
 It is believed that for a given $n$, only a small subset of the set of all graphs of order $n$ possess the distance magic property. Consequently, a different approach is required rather than providing a collection of non-isomorphic graphs as input and subsequently characterizing them for the distance magic property. Instead, the focus shifts towards developing algorithms capable of dynamically constructing all distance magic graphs on $n$ vertices. By adopting this strategy, we can bypass the challenge of generating all non-isomorphic graphs. We have successfully implemented this approach.

\section{Main Results}

The magic constant of regular distance magic graphs can be easily calculated, as shown in the following theorem.
\begin{theorem} \cite{sigma_jinnah, dm_miller, rtv, vilfredt} \label{regk}
If $G$ is an $r$-regular distance magic graph on $n$ vertices with magic constant $k$, then $k = \frac{r(n+1)}{2}$.
\end{theorem}

It is known that each positive integer of the form $4t+1$ is a magic constant of a graph union of $t$ copies of a $4$-cycle $tC_4$ as stated in the following theorem.
\begin{theorem} \cite{sigma_jinnah} \label{odd1}
A graph $G$ of order $n$ is a distance magic graph with magic constant $k = n+1$ if and only if $G = tC_4$, $t \geq 1$.
\end{theorem}
Figure \ref{fig:tc4} shows the distance magic labeling of $3C_4$ with the magic constant $13$.
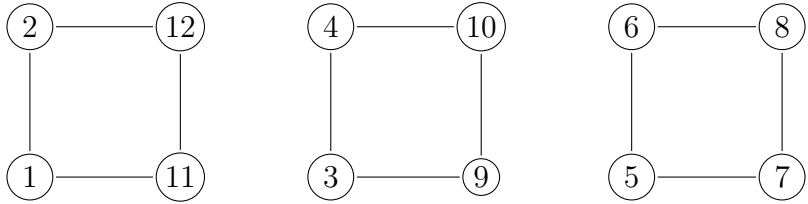
\begin{figure}[ht]
    \centering
    \begin{tikzpicture}
    \node[draw, circle, inner sep=2.7pt, minimum size=2pt] (a) at (0,0) {$1$};
    \node[draw, circle, inner sep=2.7pt, minimum size=2pt] (b) at (0,2) {$2$};
    \node[draw, circle, inner sep=1.5pt, minimum size=2pt] (c) at (2,2) {$12$};
    \node[draw, circle, inner sep=1.5pt, minimum size=2pt] (d) at (2,0) {$11$};
    \node[draw, circle, inner sep=2.7pt, minimum size=2pt] (e) at (4,0) {$3$};
    \node[draw, circle, inner sep=2.7pt, minimum size=2pt] (f) at (4,2) {$4$};
    \node[draw, circle, inner sep=1.5pt, minimum size=2pt] (g) at (6,2) {$10$};
    \node[draw, circle, inner sep=1.5pt, minimum size=2pt] (h) at (6,0) {$9$};
    \node[draw, circle, inner sep=2.7pt, minimum size=2pt] (i) at (8,0) {$5$};
    \node[draw, circle, inner sep=2.7pt, minimum size=2pt] (j) at (8,2) {$6$};
    \node[draw, circle, inner sep=2.7pt, minimum size=2pt] (k) at (10,2) {$8$};
    \node[draw, circle, inner sep=2.7pt, minimum size=2pt] (l) at (10,0) {$7$};
    
    \draw[shorten <= 1pt, shorten >= 1pt] (a) -- (b);
    \draw[shorten <= 1pt, shorten >= 1pt] (b) -- (c);
    \draw[shorten <= 1pt, shorten >= 1pt] (c) -- (d);
    \draw[shorten <= 1pt, shorten >= 1pt] (d) -- (a);
    \draw[shorten <= 1pt, shorten >= 1pt] (e) -- (f);
    \draw[shorten <= 1pt, shorten >= 1pt] (f) -- (g);
    \draw[shorten <= 1pt, shorten >= 1pt] (g) -- (h);
    \draw[shorten <= 1pt, shorten >= 1pt] (h) -- (e);
    \draw[shorten <= 1pt, shorten >= 1pt] (i) -- (j);
    \draw[shorten <= 1pt, shorten >= 1pt] (j) -- (k);
    \draw[shorten <= 1pt, shorten >= 1pt] (k) -- (l);
    \draw[shorten <= 1pt, shorten >= 1pt] (l) -- (i);
    \end{tikzpicture}
    \caption{ A distance magic labeling of $3C_4$.}
    \label{fig:tc4}
\end{figure}

The following result shows that each positive integer of the form $4t+3$ is a magic constant of a graph $P_3 \cup tC_4$. 
\begin{theorem} \cite{sigma_jinnah} \label{odd2}
Let $G$ be a distance magic graph with magic constant $k$. Then, the following are equivalent.
\begin{enumerate}
    \item $k=n$.
    \item $\delta(G) = 1$.
    \item Either $G$ is isomorphic to $P_3$ or $G$ contains exactly one copy of $P_3$ and all other components are isomorphic to $C_4$.
\end{enumerate}
\end{theorem}
Figure \ref{fig:p3tc4} shows the distance magic labeling of $P_3 \cup 2C_4$ with magic constant $11$.
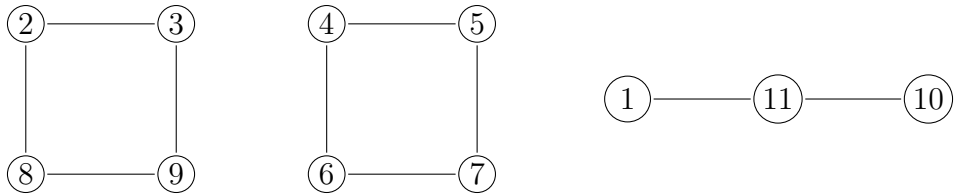
\begin{figure}[ht]
    \centering
    \begin{tikzpicture}
    \node[draw, circle, inner sep=1.5pt, minimum size=2pt] (a) at (0,0) {$8$};
    \node[draw, circle, inner sep=1.5pt, minimum size=2pt] (b) at (0,2) {$2$};
    \node[draw, circle, inner sep=1.5pt, minimum size=2pt] (c) at (2,2) {$3$};
    \node[draw, circle, inner sep=1.5pt, minimum size=2pt] (d) at (2,0) {$9$};
    \node[draw, circle, inner sep=1.5pt, minimum size=2pt] (e) at (4,0) {$6$};
    \node[draw, circle, inner sep=1.5pt, minimum size=2pt] (f) at (4,2) {$4$};
    \node[draw, circle, inner sep=1.5pt, minimum size=2pt] (g) at (6,2) {$5$};
    \node[draw, circle, inner sep=1.5pt, minimum size=2pt] (h) at (6,0) {$7$};
    \node[draw, circle, inner sep=2.7pt, minimum size=2pt] (i) at (8,1) {$1$};
    \node[draw, circle, inner sep=1.5pt, minimum size=1.5pt] (j) at (10,1) {$11$};
    \node[draw, circle, inner sep=1.5pt, minimum size=2pt] (k) at (12,1) {$10$};
    \draw[shorten <= 1pt, shorten >= 1pt] (a) -- (b);
    \draw[shorten <= 1pt, shorten >= 1pt] (b) -- (c);
    \draw[shorten <= 1pt, shorten >= 1pt] (c) -- (d);
    \draw[shorten <= 1pt, shorten >= 1pt] (d) -- (a);
    \draw[shorten <= 1pt, shorten >= 1pt] (e) -- (f);
    \draw[shorten <= 1pt, shorten >= 1pt] (f) -- (g);
    \draw[shorten <= 1pt, shorten >= 1pt] (g) -- (h);
    \draw[shorten <= 1pt, shorten >= 1pt] (h) -- (e);
    \draw[shorten <= 1pt, shorten >= 1pt] (i) -- (j);
    \draw[shorten <= 1pt, shorten >= 1pt] (j) -- (k);
    \end{tikzpicture}
    \caption{A distance magic labeling of $P_3 \cup 2C_4$.}
    \label{fig:p3tc4}
\end{figure}

From Theorem \ref{odd1} and \ref{odd2}, we conclude that all odd positive integers $\ge 3$ are magic constants. This also answers the question: does $p^{t}, t\ge 1$ belong to the set of magic constants, where $p$ is odd prime? (see \cite{kamatchit}). As a special case of odd magic constants, we have found a direct construction for graphs with magic constant  $p^{t}$, where $p$ is an odd prime. We describe that next.
\begin{proposition}
Given an odd prime $p$, $p^t$ is a magic constant for all $t \ge 1$. 
\end{proposition}
\begin{proof}
Let $p$ be an odd prime. Then $p$ is of one of the following forms.\\

\textbf{Case i.} $p \equiv 1(\bmod\ 4)$.\\
Then $p^t \equiv 1(\bmod\ 4)$, for all $t \ge 1$ and magic constant $p^t$ can be constructed by using $\frac{p^t - 1}{4}$ copies of $C_4$ as given in Theorem~\ref{odd1}.\\

\textbf{Case ii.} $p \equiv -1(\bmod\ 4)$.\\
Then $p^t \equiv 1(\bmod\ 4)$ when $t$ is even and $p^t \equiv -1(\bmod\ 4)$ when $t$ is odd. When $t$ is even, we follow the argument as in \textit{case i} above. Suppose $t$ is odd. If $t=1$, we take $G = P_3 \cup \lambda C_4$, where $\lambda = \frac{p-3}{4}$. Hence, by the Theorem \ref{odd2}, $k = p$.\\

Now, suppose $t \ge 3$. Consider a $p$-regular graph $G$ on $n_1 = \frac{p^{t-1}-1}{2}$ vertices. Then $G[\overline{K}_{2}]$ is a $2p$-regular graph on $n = 2n_1 = p^{t-1}-1$ vertices. Therefore by Theorem \ref{regk}, $k = \frac{2p((p^{t-1}-1)+1)}{2} = p^t$ as required.
\end{proof}

Now, we explore the case of the even positive integers, which are magic constant.\\
\smallskip
When $H$ is an arbitrary graph with vertices $x_1,x_2 \dots, x_n$, and $G$ is any graph with $t$ vertices, then by $H[G]$ we denote the graph, which arises from $H$ by replacing each vertex $x_i$ by a copy of the graph $G$ with vertex set $X_i = \{x_{i_1}, \dots, x_{i_t}\}$, and each edge $x_ix_j$ by the edges of the complete bipartite graph $K_{t,t}$ with bipartition $X_i,X_j$. The graph $H[G]$ is then called \textit{lexicographic product} or the \textit{composition} of $H$ and $G$.
\begin{theorem} \cite{dm_aloysius, dm_miller} \label{compo}
If $G$ is an $r$-regular graph, then $G[\overline{K}_t]$ is distance magic for any even $t$.
\end{theorem}
In \cite{dm_miller}, the authors proved that for an $r$-regular graph $G$ on $n$ vertices, the magic constant of $G[\overline{K}_t]$ is $k = rt(2nt + 1)$. With $r = 2$ and $k = 1$, we obtain $k = 4n+2$. This gives the proof of the following theorem.

\begin{theorem} \label{even1}
For all $t \ge 3$, $4t+2$ is a magic constant.
\end{theorem}

In \cite{dm_4reg_petr}, the authors proved that there exists a $4$-regular distance magic graph $G=(V,E)$ of odd order if and only if $|V| \ge 17$. Now, we use this result to prove the existence of even magic constants of the form $4t+4$, $t \ge 8$.
\begin{theorem} \label{even2}
For all $t \ge 8$, $4t+4$ is a magic constant.
\end{theorem}
\begin{proof}
 Let $G$ be a $4$-regular distance magic graph of odd order. Then, $|V(G)| \ge 17$. By Theorem \ref{regk}, the magic constant $k$ of such a graph $G$ is given by $k = \frac{4(|V|+1)}{2} = 2|V|+2$. If we write $|V| = 2t+1$, where $t \ge 8$, then $k = 4t+4$. This proves the theorem.
\end{proof}
We already know that the integers $1, 2, 4, 6, 8$ and $12$ are not magic constants \cite{kamatchit}, while $10$  and $24$  are magic constants (see, \cite{kamatchit}, \cite{dm_algo_fuad}). Theorems \ref{even1} and \ref{even2} provide further insights into the nature of the set of magic constants. Theorem \ref{even1} reveals that all positive integers congruent to $2(\bmod\ 4)$, with the exceptions of $2$ and $6$, are the members of the set of magic constants. Meanwhile, Theorem \ref{even2} establishes that all positive integers greater than or equal to $36$ and congruent to $0(\bmod\ 4)$ are also in the set of magic constants. 

As a result, the only remaining even positive integers requiring examination as magic constants are $16, 20, 28$, and $32$. We device an algorithmic approach to determine whether these integers qualify as magic constants of some graphs. Thus we have a complete characterisation of all magic constants (see Theorem \ref{conclusion}). 

\section{An algorithm for constructing magic graphs and computing magic constants}
Given positive integers $n$ and $k$, our algorithm checks if there exists a graph with vertex set $\{1,2, \dots, n\}$ and a magic constant $k$. We denote any such graphs as $G(n,k)$. Without loss of generality, we may assume that each vertex $i$ is labeled as $i$ for $1 \le i \le n$. The algorithm returns the first successful search if such a magic graph exists.
\smallskip
\par Let $S$ be the set of all subsets of the set $\{1,2, \dots, n\}$ with the sum of its elements equal to $k$. For each $1 \le i \le n$, let $NS(i) = \{T \in S: i \not \in T\}$. Note that if there exists a distance magic graph $G(n,k)$, then the neighborhood of a vertex $i$, $N(i)$, must be a member of $NS(i)$. Next, fix a neighborhood of a vertex $n$, i.e., an element of $NS(n)$, and construct a tree rooted at $N(n)$ in the following manner. Add all elements in $NS(n-1)$ as children to $N(n)$. This is level $1$. Add elements of $NS(n-2)$ as children to each node in level $1$. This is level $2$. Continue till we add all elements in $NS(1)$ to each node in level $n-1$. Note that in level $i$, each node of this tree $T(n,k)$ is a possible candidate for a $N(i)$. Each branch of this tree $T(n,k)$ gives an adjacency list for the vertices $1,2, \dots, n$. These lists may or may not correspond to a graph. Whenever adjacency lists corresponding to a branch give a graph, we call such a branch a successful branch.

The algorithm consists of the following four steps:
\subsection{Algorithm}
\begin{enumerate}
    \item[Step 1.] Generate all subsets of the set $\{1,2, \dots, n\}$ that have the sum of its elements equal to $k$. 
    \item[Step 2.] For each $1 \le i \le n$, construct $NS(i)$.
    \item[Step 3.] Construct a rooted tree $T(n,k)$ with a root as a fixed element of $NS(n)$.
    \item[Step 4] For a tree $T(n,k)$ obtained in Step ($3$), find a successful branch. Repeat Steps ($3$) and ($4$) for each element of $NS(n)$ until we find a successful branch, or we run out of elements in $NS(n)$.
\end{enumerate}

The approach described above is a brute force in some sense. For example, for the case $n = 30$, $k = 32$, there are more than $10^{76}$ branches to explore. We merge Step $(3)$ and Step $(4)$ to say Step ($3'$) to optimise the search space as:
\begin{enumerate} 
    \item[Step $3'$] We explore the branches of $T(n,k)$ as discussed in Step $(3)$ and Step $(4)$ with the following condition. Let the tree be constructed till level $i (2 \le i \le n-1)$ and $NS(i) = \{s_1, s_2, \dots, s_{t_i}\}$, $NS(i+1) = \{r_1, r_2, \dots, r_{t_{h}}\}$. Next, while adding children $r_b$ to any of the $s_a$ we will check 
    \begin{equation} \label{eq: check1}
    i+1 \text{ is in any of it's parents } s_a \text{ if and only if } i \in r_b.
    \end{equation}
    Also, at any moment, 
    \begin{equation} \label{eq: check2}
        \text{neighborhood sum of any vertex does not exceed } k.
    \end{equation}
\end{enumerate}

Step ($3'$) ensures symmetry, meaning that $u$ is adjacent to $v$ if and only if $v$ is adjacent to $u$. If the condition stated in \ref{eq: check1} or \ref{eq: check2} is not met at any step, we avoid adding the corresponding node, preventing further growth of the tree below that node. This reduction strategy significantly reduces the size of the tree $T(n,k)$, leading to an exponential reduction in the search space.

\subsection{Proof of correctness}
It is sufficient to prove that every $G(n,k)$ graph corresponds to a (successful) branch of some $T(n,k)$ and each successful branch of each $T(n,k)$ corresponds to some graph $G(n,k)$. It is easy to observe that if $T(n,k)$ has a successful branch, then nodes on that branch constitute a graph $G(n,k)$.

\smallskip
Conversely, suppose there is a graph $G(n,k)$. Consider a tree $T(n,k)$ rooted at $N(n)$. There is a one-to-one correspondence between branches of $T(n,k)$ and the elements of the cartesian product $N(n) \times NS(n-1) \times \dots \times NS(1)$. The neighborhood of each vertex $i$, $N(i)$, is a member of $NS(i)$. Hence, $N(n) \times N(n-1) \times \dots \times N(1)$ is an element of the cartesian product $N(n) \times NS(n-1) \times \dots \times NS(1)$. Thus, it corresponds to some branch of $T(n,k)$.

\subsection{Illustration}
Let us illustrate these steps with an example. Let $n = 7$ and $k = 7$. We know that there is only one graph $P_3 \cup C_4$ up to isomorphism on $7$ vertices admitting magic constant $7$ \cite{dm_algo_fuad}. All subsets of the set $\{1,2, \dots, 7\}$ having sum $7$ are:
\begin{equation} \nonumber
    \bigr[[7],\; [6, 1],\; [5, 2],\; [4, 3],\; [4, 2, 1] \bigr].
\end{equation}
We have listed all the subsets as a list of lists for the proper indexing purpose in the pseudo-code. This completes Step $1$. Next, we generate all possible neighborhood sets for each vertex as suggested in Step $2$.
\begin{align*} \nonumber
&NS(1) = \bigr[[7],\; [5, 2],\; [4, 3] \bigr]\\
&NS(2) = \bigr[[7],\; [6, 1]\; [4, 3] \bigr]\\
&NS(3) = \bigr[[7],\; [6, 1],\; [5, 2],\; [4,2,1] \bigr]\\
&NS(4) = \bigr[[7],\; [6,1],\; [5, 2] \bigr]\\
&NS(5) = \bigr[[7], \; [6,1],\; [4, 3]\; [4, 2, 1] \bigr]\\
&NS(6) = \bigr[[7],\; [5, 2],\; [4, 3],\; [4, 2, 1] \bigr]\\
&NS(7) = \bigr[[6, 1],\; [5, 2],\; [4, 3],\; [4, 2, 1] \bigr]
\end{align*}
Now we construct a tree $T(7,7)$ rooted at the node $[6,1]$ from $NS(7)$. This makes $N(7) = [6, 1]$. Add all the elements of the set $NS(6)$ as children to the root node as shown in Figure \ref{fig: (7,7)tree}. The set of possible candidates for $N(6)$ is
\begin{equation} \nonumber
    NS(6) = \bigr[[7],\; [5, 2],\; [4, 3],\; [4, 2, 1] \bigr]
\end{equation}
We have chosen $NS(7) = [6, 1]$, to maintain symmetry we must have $7 \in N(6)$. The only such element in $NS(6)$ is $[7]$. Hence, we discard all other nodes in the level $1$, and the tree won't grow further below those nodes. Now we go to level $2$. The set of possible candidates for $N(5)$ is
\begin{equation} \nonumber
    NS(5) = \bigr[[7], \; [6,1],\; [4, 3]\; [4, 2, 1] \bigr]
\end{equation}
So far we have $N(7) = [6, 1]$ and $N(6) = [7]$. Since $5 \not \in N(6) \cup N(7)$, $6, 7 \not \in N(5)$ maintaining the condition given in (\ref{eq: check1}) of Step $3'$. Then the only choices for $N(5)$ are $[4, 3]$ and $[4, 2, 1]$. Without loss of generality consider $N(5) = [4,3]$. 
Now we go to level $3$. The set of possible candidates for $N(4)$ is
\begin{equation} \nonumber
    NS(4) = \bigr[[7],\; [6,1],\; [5, 2] \bigr]
\end{equation}
So far we have $N(7) = [6, 1]$, $N(6) = [7]$ and $N(5) = [4, 3]$. Here $4 \not \in N(6) \cup N(7)$ hence $6, 7 \not \in N(4)$ but $4 \in N(5)$ hence $N(4)$ must contain $5$. The only choice for $N(4)$ is $[5,2]$. We continue in a similar way, and lastly, we add pendants $[7]$, $[5,2]$, and $[4,3]$ from the set $NS(1)$. For $N(1)$ we discard the nodes $[5, 2]$ and $[4, 3]$ because $1 \not \in N(j)$ for any $j = 2, 3, 4, 5$ and we must $N(1)$ must have $7$ since $1 \in N(7)$ hence maintaining symmetry condition as given in (\ref{eq: check1}) of Step $3'$.\\

\begin{figure}[ht]
\centering
\begin{tikzpicture}[
    level 1/.style={sibling distance = 2cm, level distance = 2cm},
    level 2/.style={sibling distance = 2cm},
    level 3/.style={sibling distance = 2cm, level distance = 2cm},
    level 4/.style={sibling distance = 2cm},
    level 5/.style={sibling distance = 2cm},
    level 6/.style={sibling distance = 2cm},
    level 7/.style={sibling distance = 2cm}
][scale = 0.5]
\node[ ] (root) {$[6, 1]$}
    child {node[] {$[7]$} 
        child {node[] {$[7]$}}
        child {node[] {$[6, 1]$}}
        child {node[] {$[4, 3]$}
            child {node[] {$[7]$}}
            child {node[] {$[6, 1]$}}
            child {node[] {$[5, 2]$}
                child {node[] {$[7]$}}
                child {node[] {$[6, 1]$}}
                child {node[] {$[5, 2]$}
                    child {node[] {$[7]$}}
                    child {node[] {$[6, 1]$}}
                    child {node[] {$[4, 3]$}
                        child {node[] {$[7]$}}
                        child {node[] {$[5, 2]$}}
                        child {node[] {$[4, 3]$}}
                    }
                }
            child {node[] {$[4, 2, 1]$}}
                }
            }
        child {node[] {$[4, 2, 1]$}
            child {node [right = 2 cm] {$[5, 2]$}
            child {node [right = 1.3cm] {\tiny{\textcolor{red}{NOT POSSIBLE}}}}
            }
        }
        }
    child {node[] {$[5,2]$}}
    child {node[] {$[4, 3]$}}
    child {node[] {$[4, 2, 1]$}};
\draw[dotted] (root-2) -- ++(0,-1) node[midway] {};
\draw[dotted] (root-3) -- ++(0,-1) node[midway] {};
\draw[dotted] (root-4) -- ++(0,-1) node[midway] {};
\draw[dotted] (root-1-1) -- ++(0,-0.7) node[midway] {};
\draw[dotted] (root-1-2) -- ++(0,-0.7) node[midway] {};
\draw[dotted] (root-1-3-1) -- ++(0,-0.7) node[midway] {};
\draw[dotted] (root-1-3-2) -- ++(0,-0.7) node[midway] {};
\draw[dotted] (root-1-3-3-1) -- ++(0,-0.7) node[midway] {};
\draw[dotted] (root-1-3-3-2) -- ++(0,-0.7) node[midway] {};
\draw[dotted] (root-1-3-3-4) -- ++(0,-0.7) node[midway] {};
\draw[dotted] (root-1-3-3-3-1) -- ++(0,-0.7) node[midway] {};
\draw[dotted] (root-1-3-3-3-2) -- ++(0,-0.7) node[midway] {};

\draw[color = blue] (root) -- (root-1);
\draw[color = blue] (root-1) -- (root-1-3);
\draw[color = blue] (root-1-3) -- (root-1-3-3);
\draw[color = blue] (root-1-3-3) -- (root-1-3-3-3);
\draw[color = blue] (root-1-3-3-3) -- (root-1-3-3-3-3);
\draw[color = blue] (root-1-3-3-3-3) -- (root-1-3-3-3-3-1);
\draw[color = red] (root-1) -- (root-1-4);
\draw[color = red] (root-1-4) -- (root-1-4-1);
\draw[color = red] (root-1-4-1) -- (root-1-4-1-1);

\node[ ] at (-10, 0) {$N(7)$};
\node[ ] at (-10, -2) {$N(6)$};
\node[ ] at (-10, -4) {$N(5)$};
\node[ ] at (-10, -6) {$N(4)$};
\node[ ] at (-10, -8) {$N(3)$};
\node[ ] at (-10, -10) {$N(2)$};
\node[ ] at (-10, -12) {$N(1)$};
\end{tikzpicture}
\caption{A tree $T(n, k)$ rooted at $[6, 1]$ with a successful branch.}
\label{fig: (7,7)tree}
\end{figure}
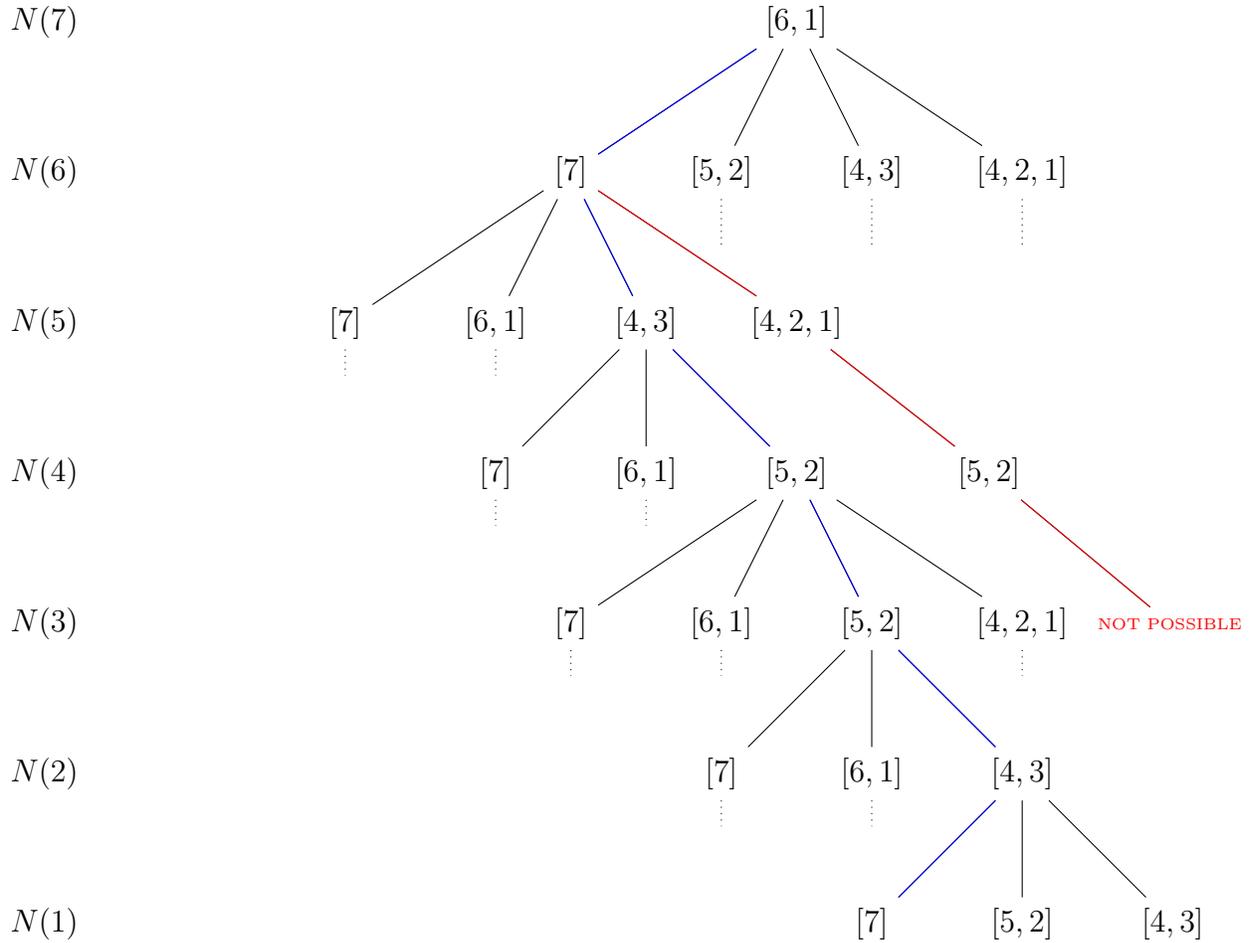

In Figure \ref{fig: (7,7)tree}, we have illustrated a case of a successful branch (colored blue). \\

Now let us consider where we select the node $[4, 2, 1]$ instead of $[4, 3]$ at level $2$. In this case so far we have $N(7)= [6, 1]$, $N(6) = [7]$ and $N(5) = [4, 2, 1]$. Then we go to level $3$. The set of possible candidates for $N(4)$ is:
\begin{equation}\nonumber
    NS(4) = \bigr[[7],\; [6,1],\; [5, 2] \bigr]
\end{equation}
Since $4 \not \in N(6) \cup N(7)$ but $4 \in N(5)$, we must have $6,7 \not \in N(4)$ but $5 \in N(4)$. The only such possibility is $[5, 2]$. We go to the next level $4$. The set of possible candidates for $N(3)$ is:
\begin{equation}\nonumber
    NS(3) = \bigr[[7],\; [6, 1],\; [5, 2],\; [4,2,1] \bigr]
\end{equation}
Since $3 \not \in N(j)$ for any $j = 4, 5, 6, 7$ and each list in $NS(3)$ contains at least one of the numbers $4, 5, 6, 7$, in order to maintain the symmetry $N(3)$ has no choices and we discard this branch. This discarded branch is shown in red color in Figure \ref{fig: (7,7)tree}.
\begin{figure}[ht]
    \centering
    \includegraphics[scale = 0.2]{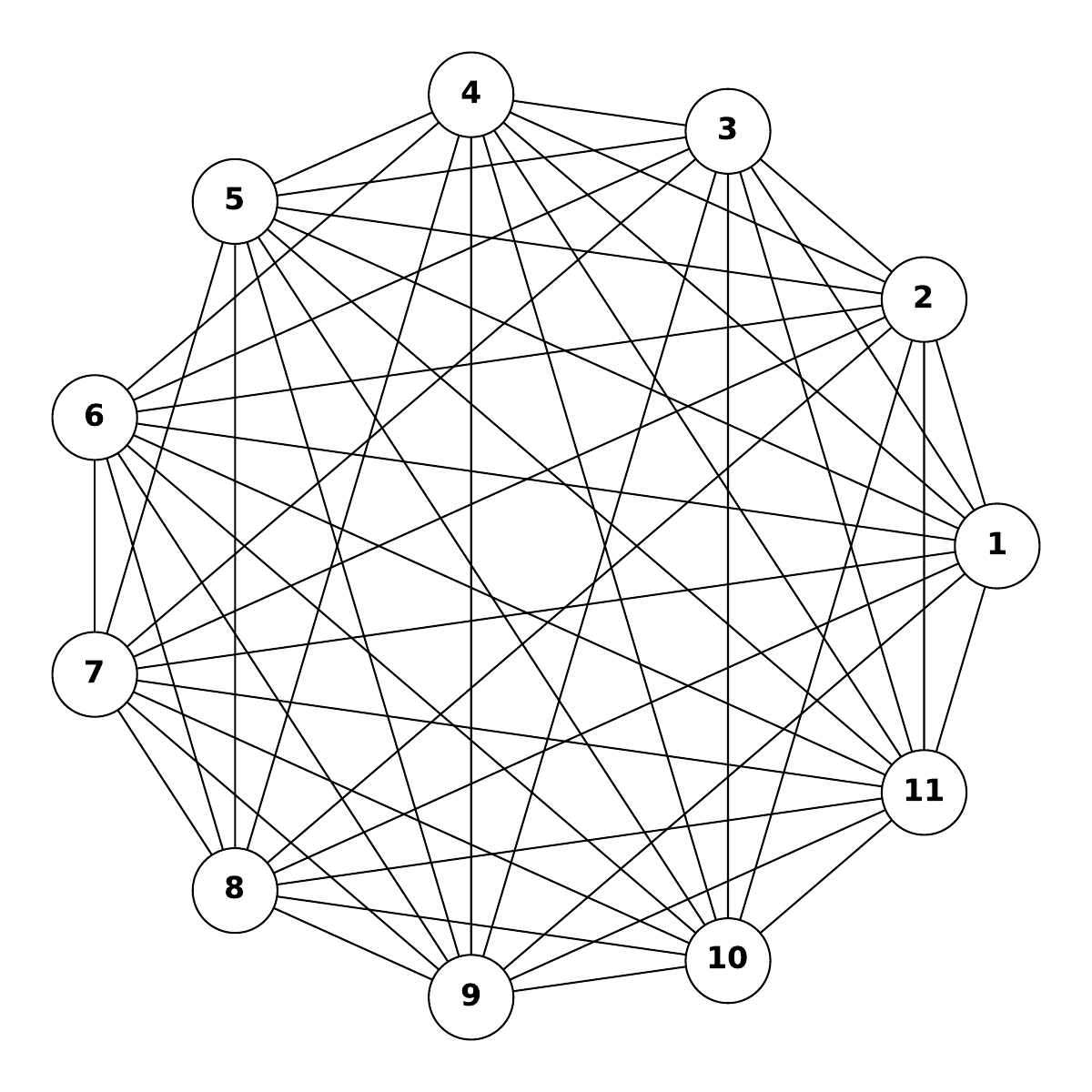}
    \caption{$G(11, 55)$.}
    \label{fig: 11,55}
\end{figure}
\subsection{Graphs output with our algorithm }
For a distance magic graph $G$ on $n$ vertices with distance magic labeling $f$ and magic constant $k$, it is easily seen that $\sum_{v \in V}deg(v)f(v) = nk$ \cite{vilfredt}. For any vertex $v$, $1 = \delta(G) \le deg(v) \le \Delta(G) = n-1$. Hence, we have an upper bound $k \le \frac{n^2-1}{2}$. This bound is sharp (Figure \ref{fig: 11,55} shows an example). Also, by definition of the distance magic graph, $k$ cannot be less than $n$. Hence, we have $n \le k \le \frac{n^2 -1 }{2}$ and both the bounds are sharp. To check whether a given positive integer $k$ is a magic constant, we need to run the algorithm for each pair $(n, k)$, where $1 \le n \le k$. Let $\alpha$ be the smallest positive integer such that $\frac{\alpha(\alpha+1)}{2} \le k$. Hence, running the algorithm for the pairs $(n, k)$, where $\alpha \le n \le k$ is sufficient. For example for $k = 16$, the possibilities for $n$ are $6, 7, \dots, 16$. Further, we can omit a few values in some cases, as mentioned in (Chapter 2, \cite{kamatchit}).\\
\smallskip
Our algorithm explores the possibility of $k = 16$ being the magic constant for the graphs on a possible number of vertices $n$. However, findings show that no graph admits the magic constant $16$. On the other hand, multiple graphs admit $20, 28, 32$ as magic constants. As a representative example, we list one of the such graphs in  Figure \ref{fig: 20,28,32}. It's not practical to list all the non-isomorphic distance magic graphs of order up to $12$ due to their abundance (more than a thousand in number, see Table \ref{tab: alldmgupto12}), we will list them of order up to $10$. Also, the collection of distance magic graphs of order up to $9$, generated using our algorithm is the same as the one given in \cite{dm_algo_fuad} except for graphs of order $8$. Hence, we list the distance magic graphs of order $8$ and $10$ only  (see figures \ref{fig: dmg8}, \ref{fig: dmg10}).

\begin{table}[ht]
		\centering
		\begin{tabular}{|c|c|c|c|c|c|c|c|c|c|c|}
		\hline
        number of vertices: $n$ & $3$ & $4$ & $5$ & $6$ & $7$ & $8$ & $9$ & $10$ & $11$ & $12$\\
        \hline
        number of non-isomorphic distance magic graphs & $1$ & $1$ & $1$ & $1$ & $3$ & $6$ & $5$ & $5$ & $74$ & $1160$\\
		\hline
       \end{tabular}	
    	\caption{Number of non-isomorphic distance magic graphs of order up to $12$}
     \label{tab: alldmgupto12}
\end{table}

\begin{figure}[ht]
  \centering
  \begin{subfigure}{0.3\textwidth}
   \includegraphics[scale = 0.25]{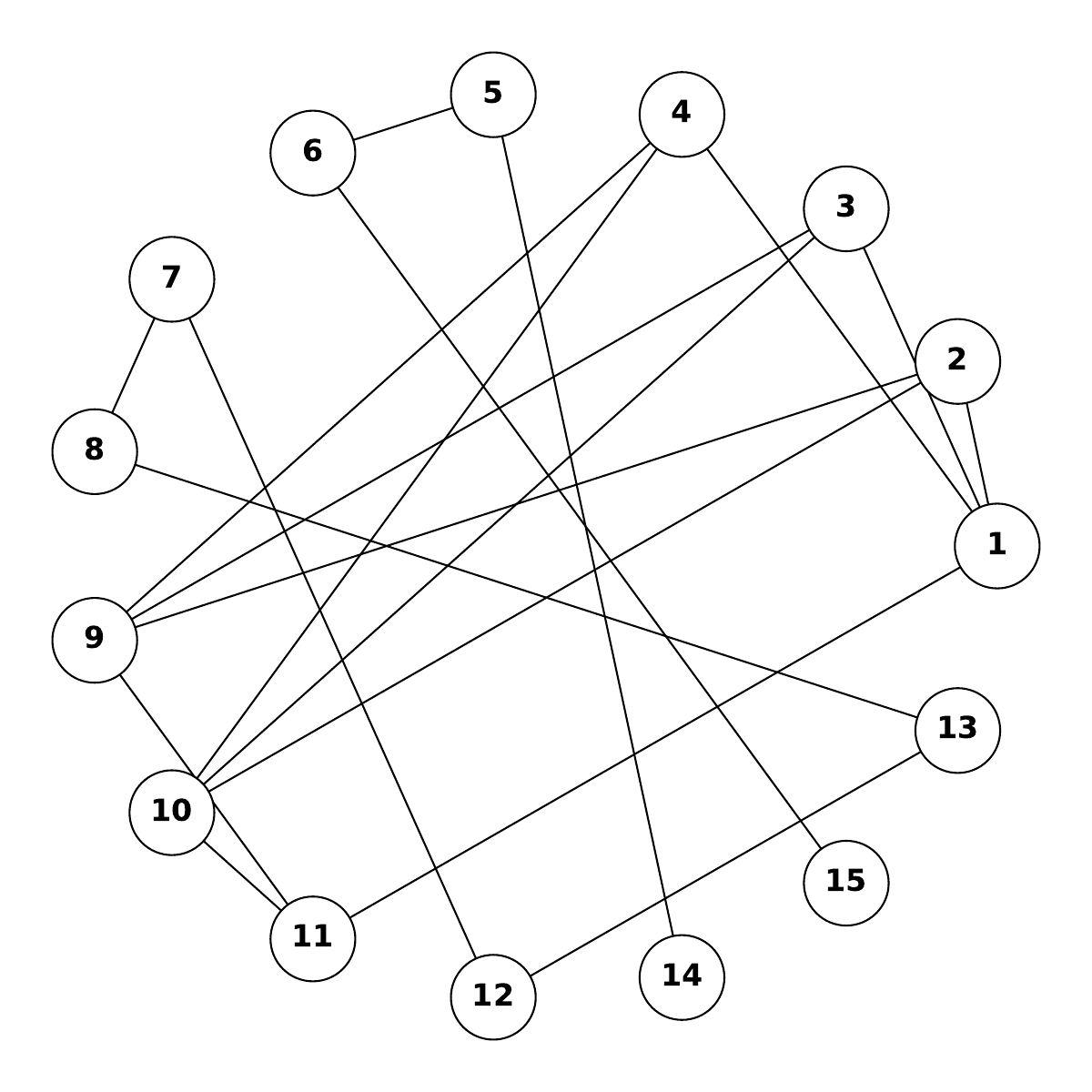}
    \caption{$G(15, 20)$}
  \end{subfigure}
  \hfill
    \begin{subfigure}{0.3\textwidth}
    \centering
    \includegraphics[scale = 0.25]{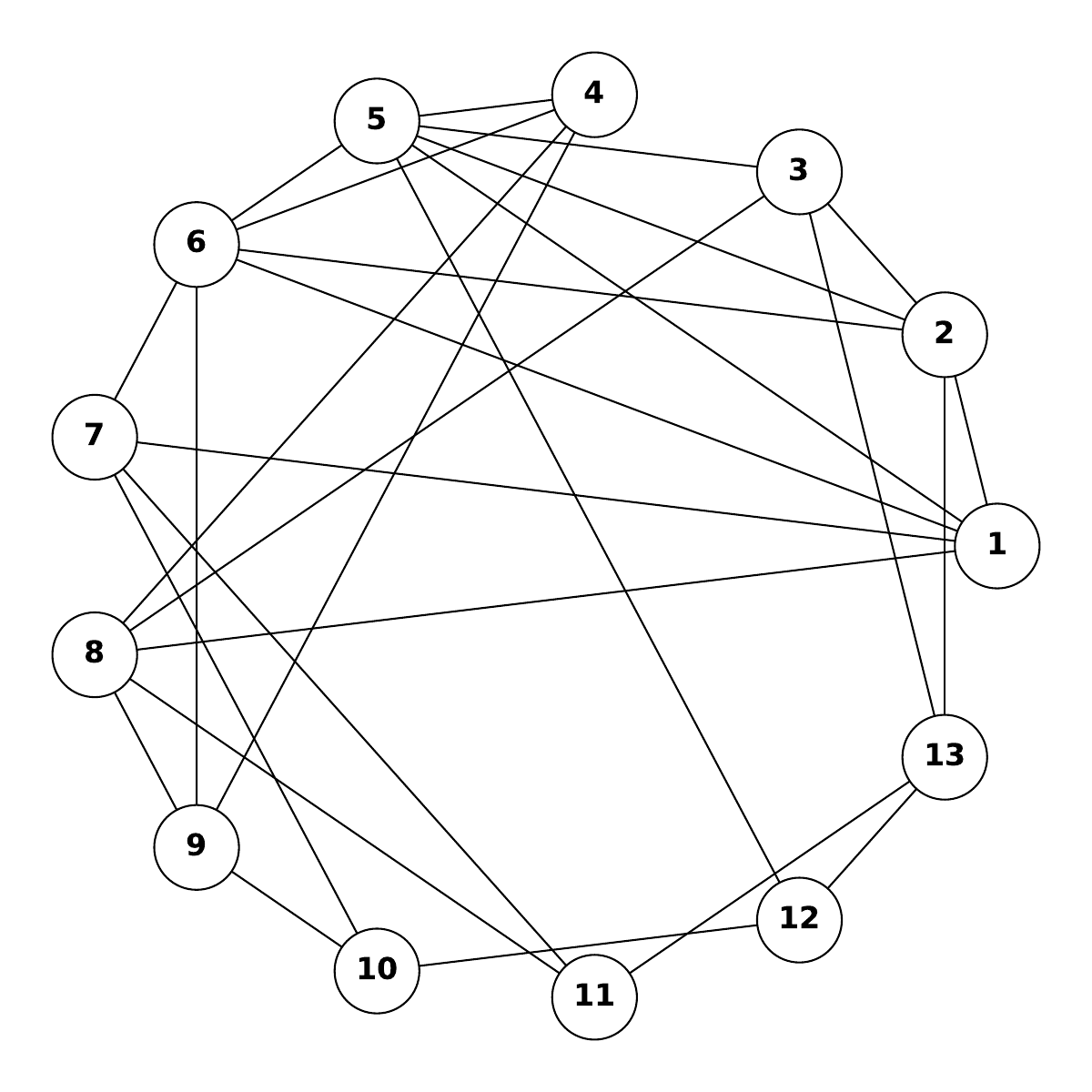}
    \caption{$G(13, 28)$}
  \end{subfigure}
  \hfill
  \begin{subfigure}{0.3\textwidth}
       \includegraphics[scale = 0.25]{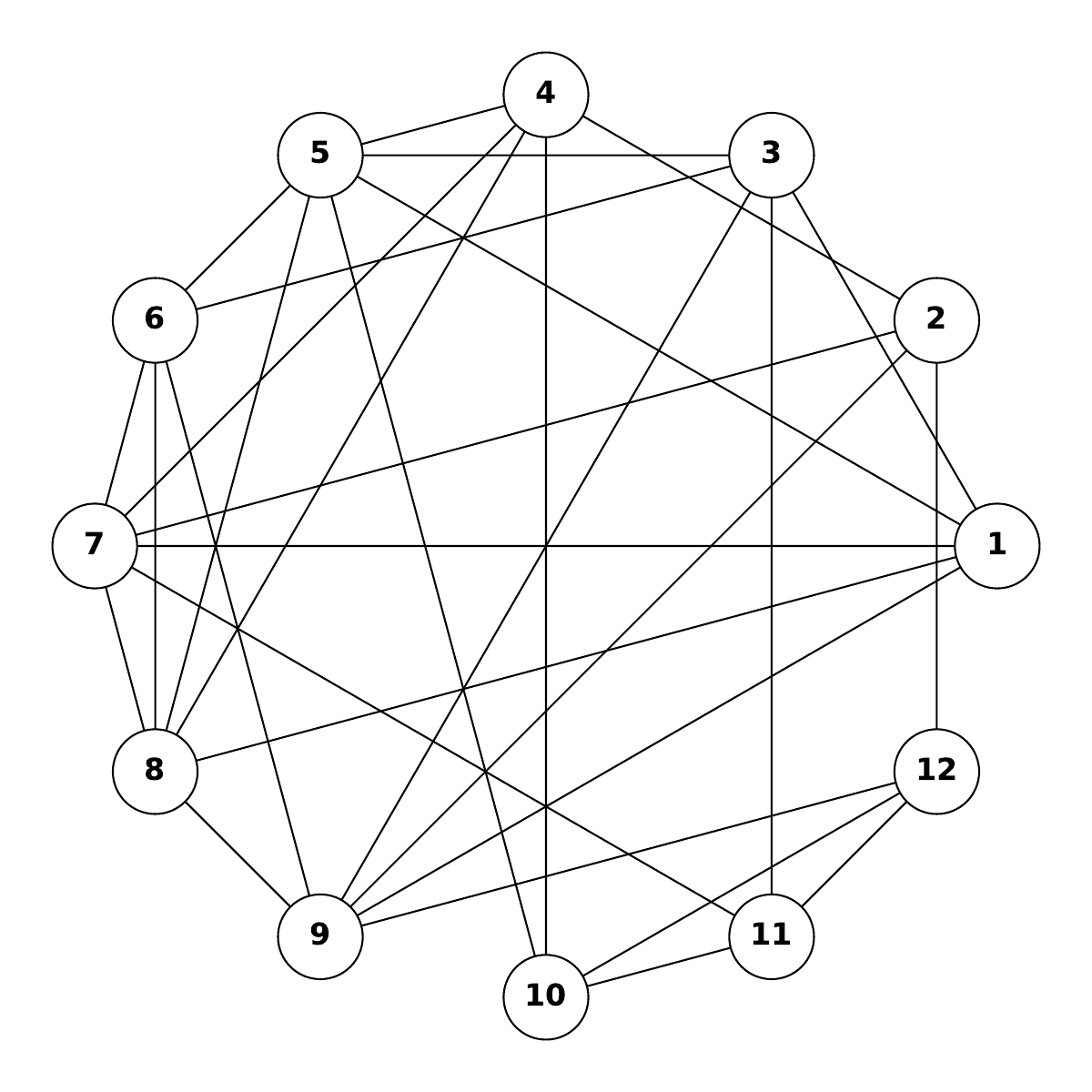}
    \caption{$G(12, 32)$}
  \end{subfigure}
  \caption{Graphs with magic constants $20, 28$ and $32$}
  \label{fig: 20,28,32}
\end{figure}

\clearpage
\begin{figure}[ht]
    \begin{subfigure}{0.4\textwidth}
        \includegraphics[scale = 0.3]{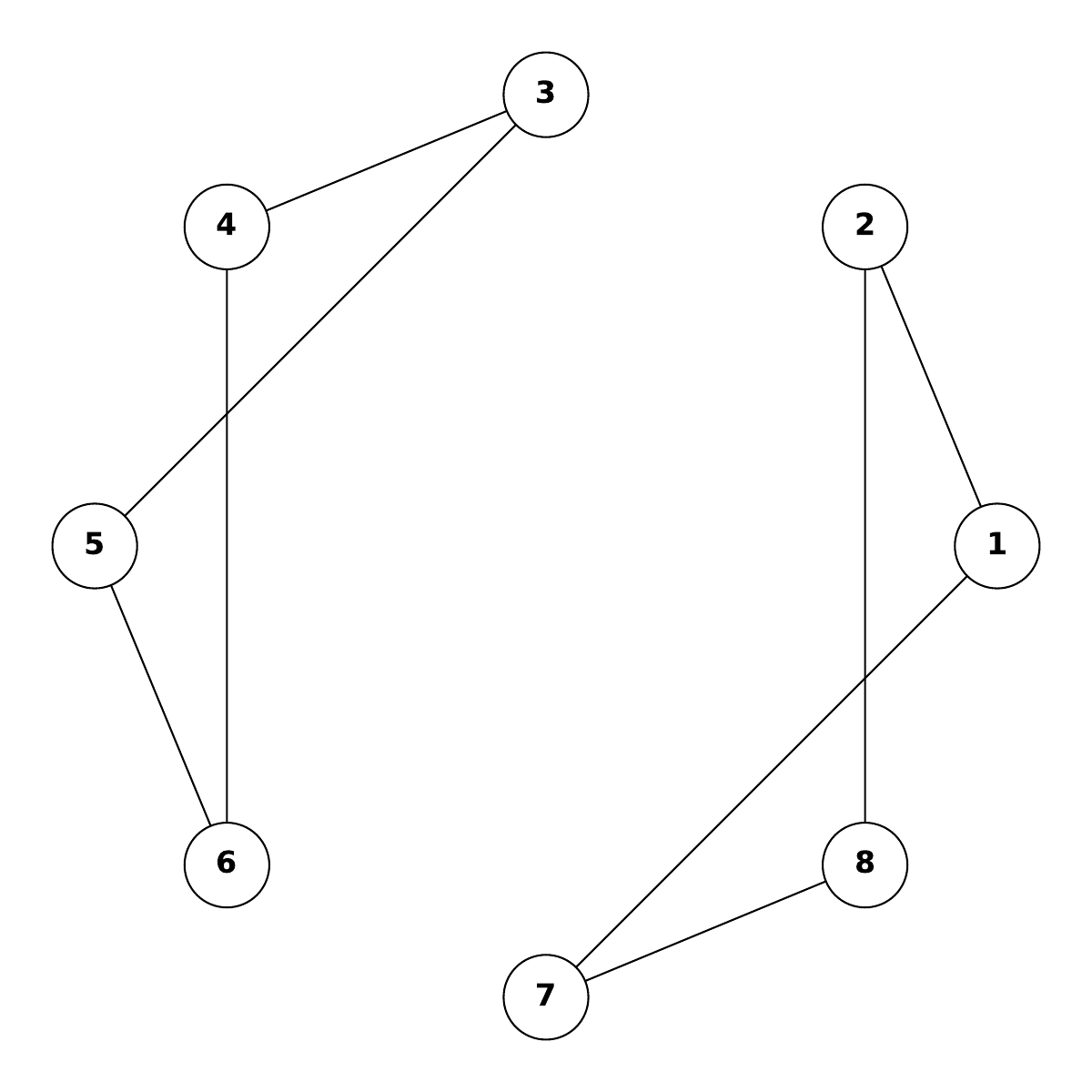}
    \end{subfigure}
    \hfill
    \begin{subfigure}{0.4\textwidth}
        \includegraphics[scale = 0.3]{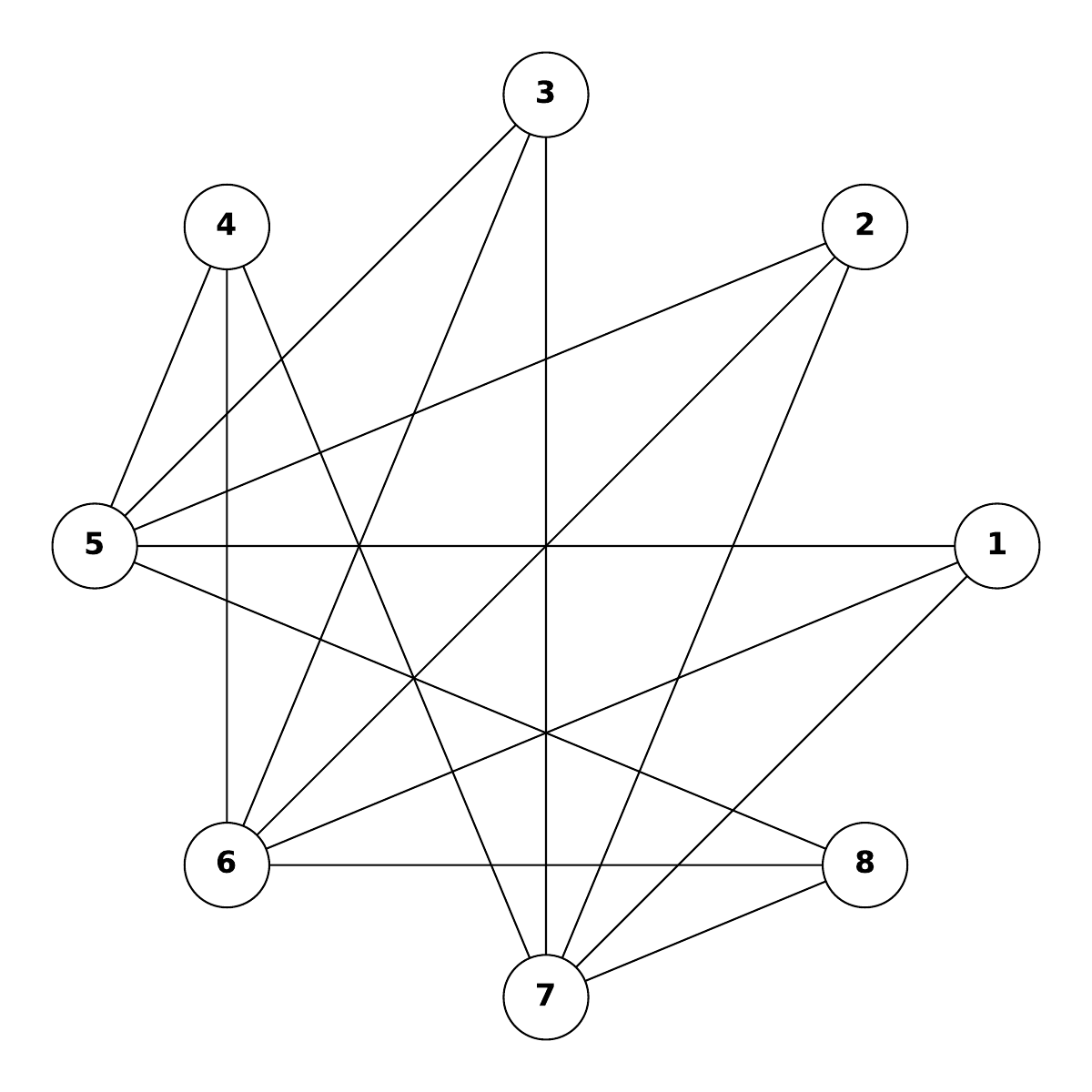}
    \end{subfigure}
    \hfill
    \begin{subfigure}{0.4\textwidth}
        \includegraphics[scale = 0.3]{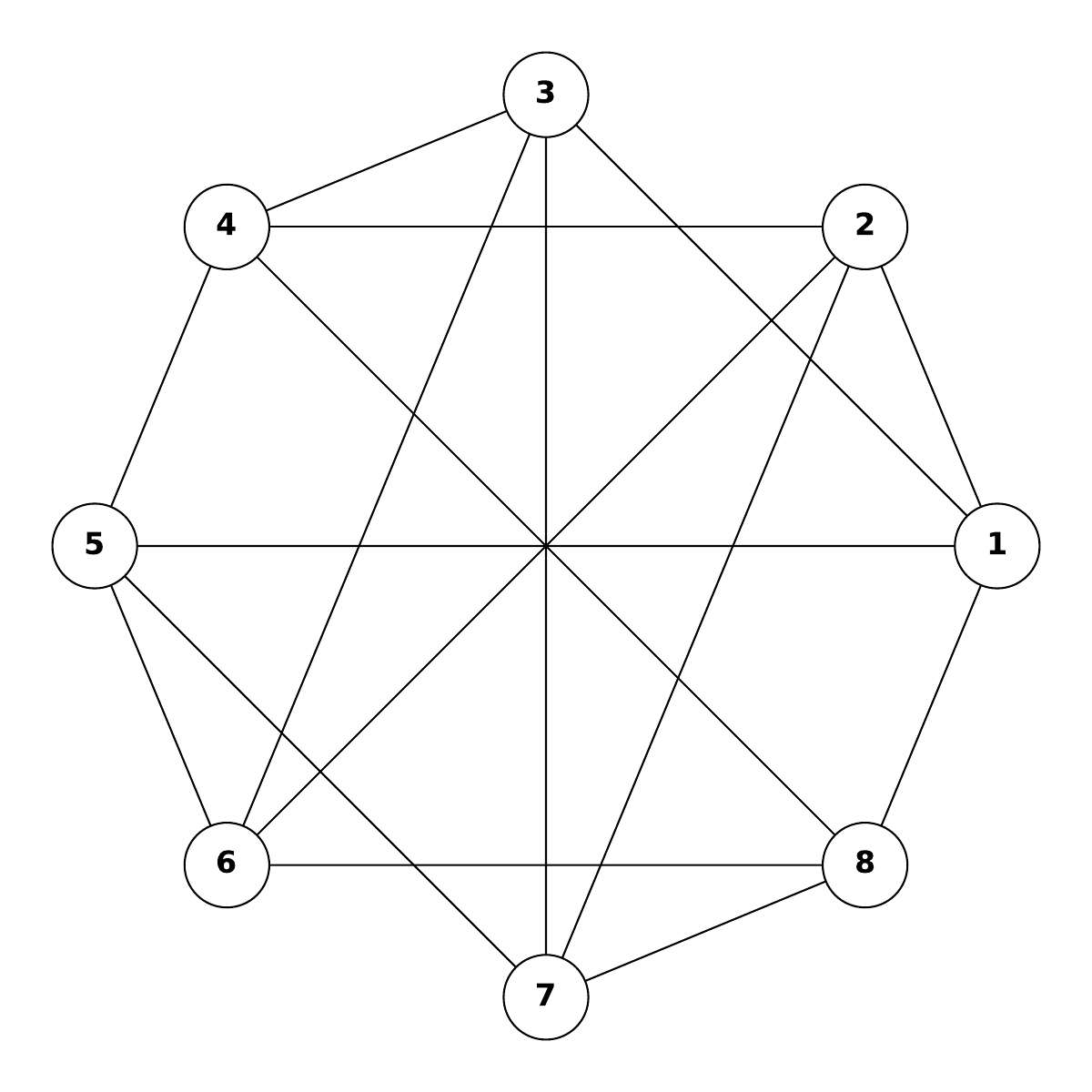}
    \end{subfigure}
    \hfill
    \begin{subfigure}{0.4\textwidth}
        \includegraphics[scale = 0.3]{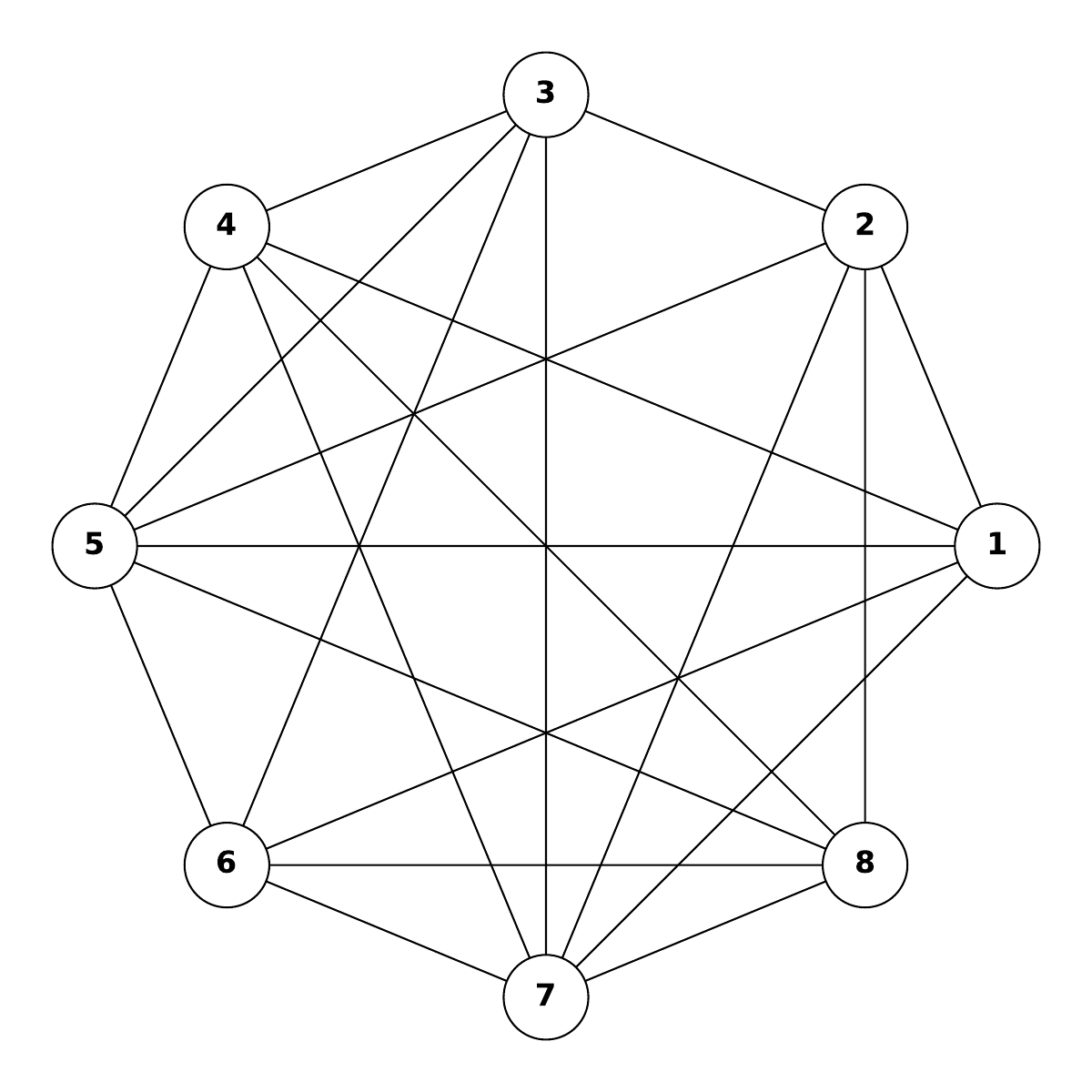}
    \end{subfigure}
    \hfill
    \begin{subfigure}{0.4\textwidth}
        \includegraphics[scale = 0.3]{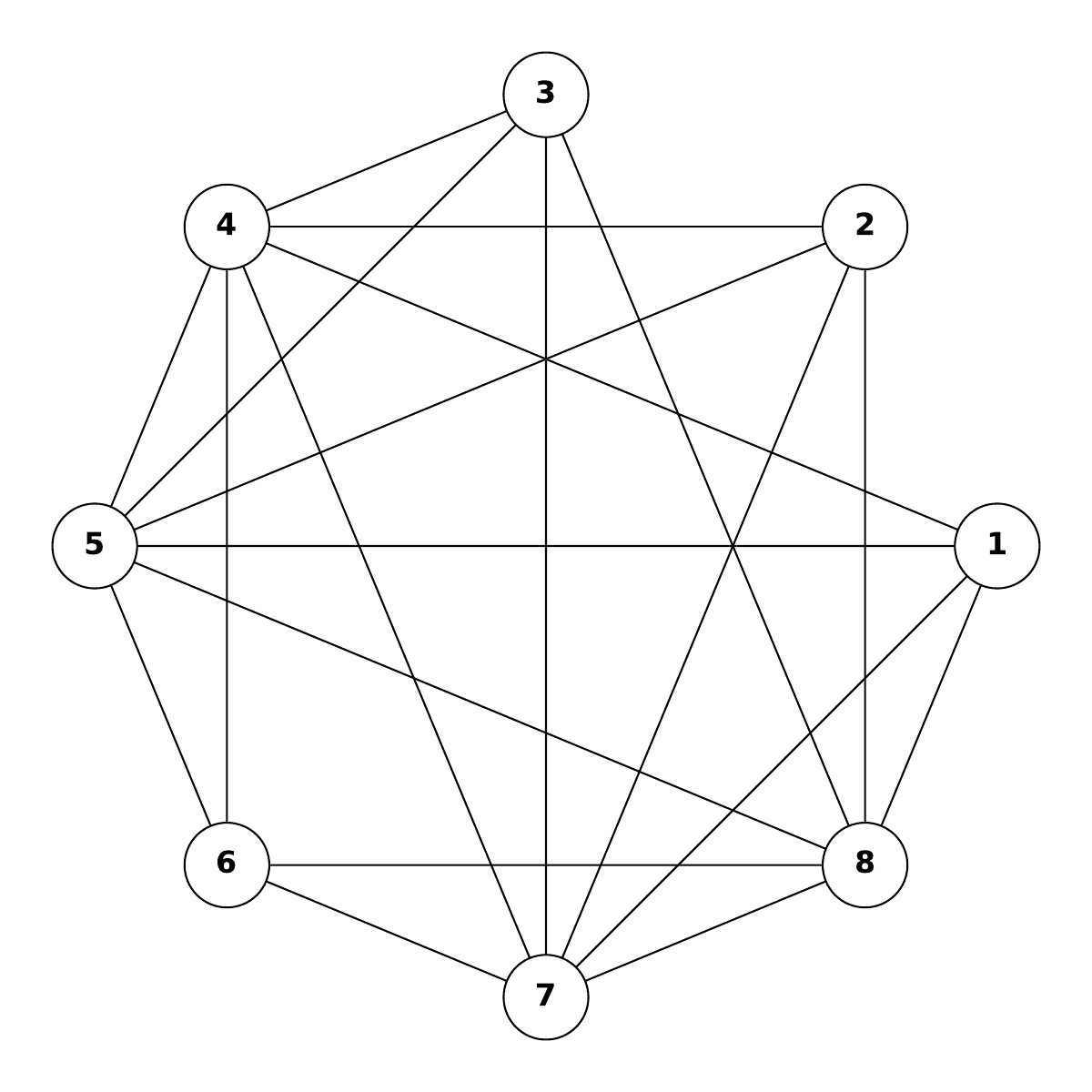}
    \end{subfigure}
    \hfill
    \begin{subfigure}{0.4\textwidth}
        \includegraphics[scale = 0.3]{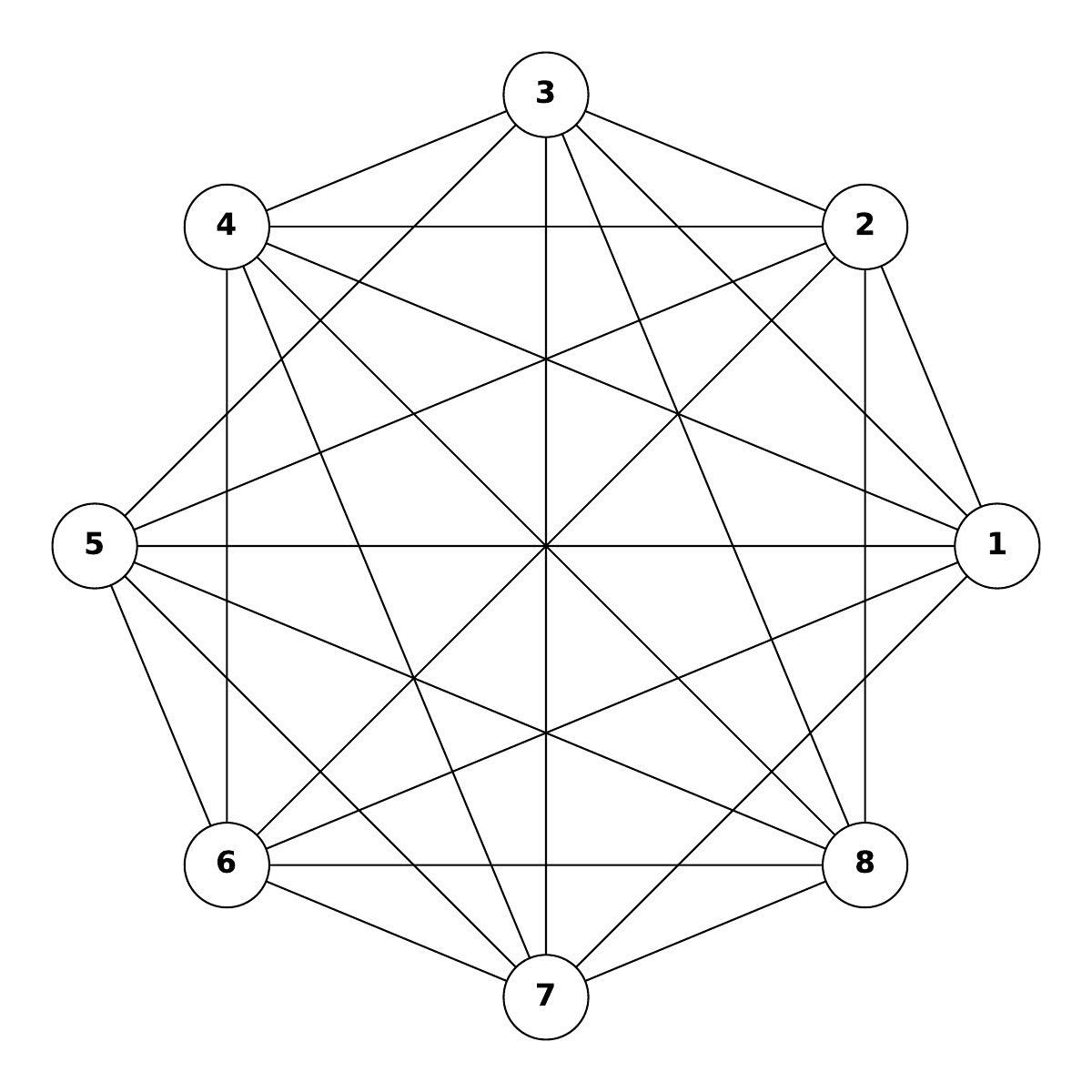}
    \end{subfigure}
\caption{Distance magic graphs of order $8$.}
\label{fig: dmg8}
\end{figure}
\clearpage
\begin{figure}[ht]
        \begin{subfigure}{0.28\textwidth}
        \includegraphics[scale=0.28]{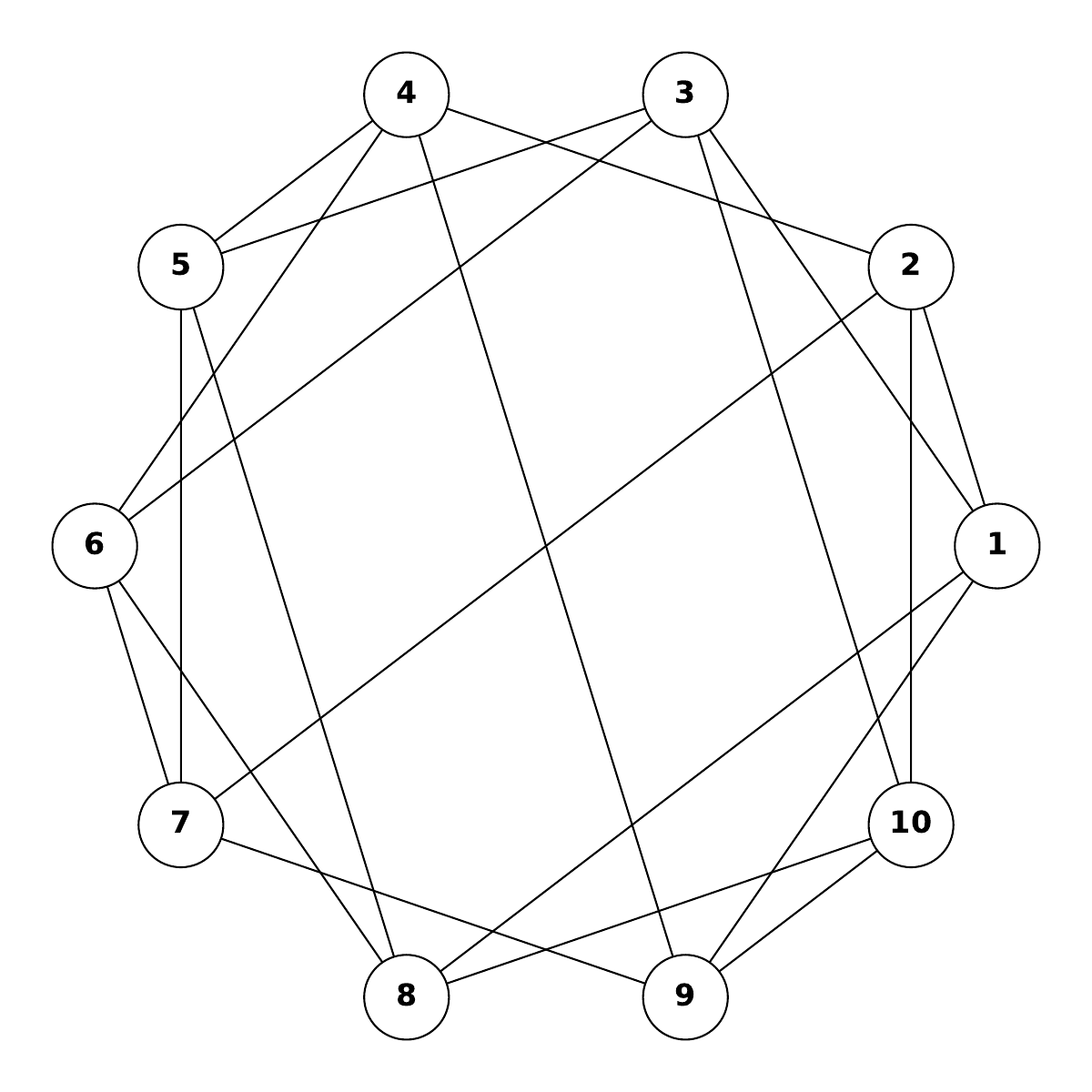}
    \end{subfigure}
    \hfill
        \begin{subfigure}{0.28\textwidth}
        \includegraphics[scale=0.28]{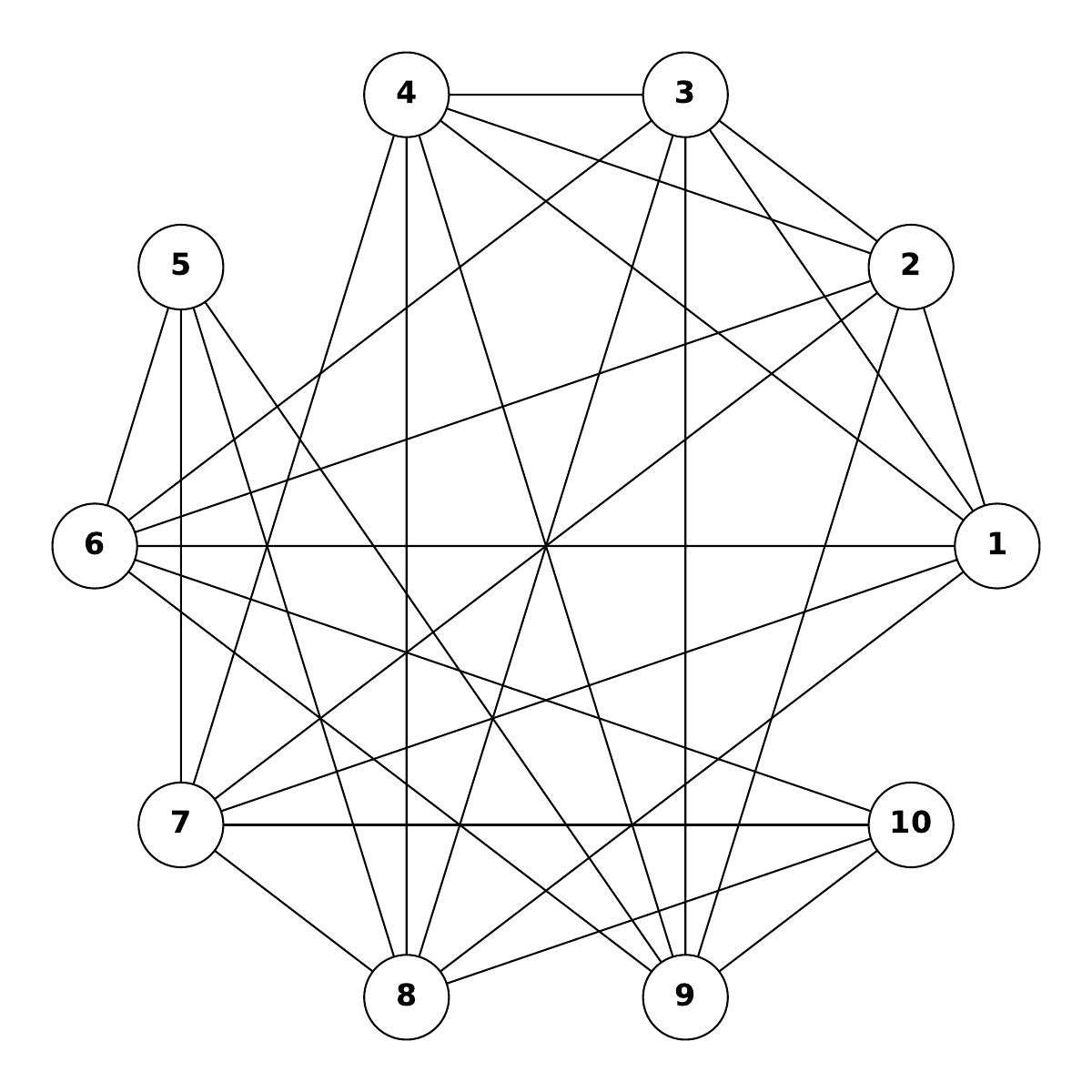}
    \end{subfigure}
    \hfill
        \begin{subfigure}{0.28\textwidth}
        \includegraphics[scale=0.28]{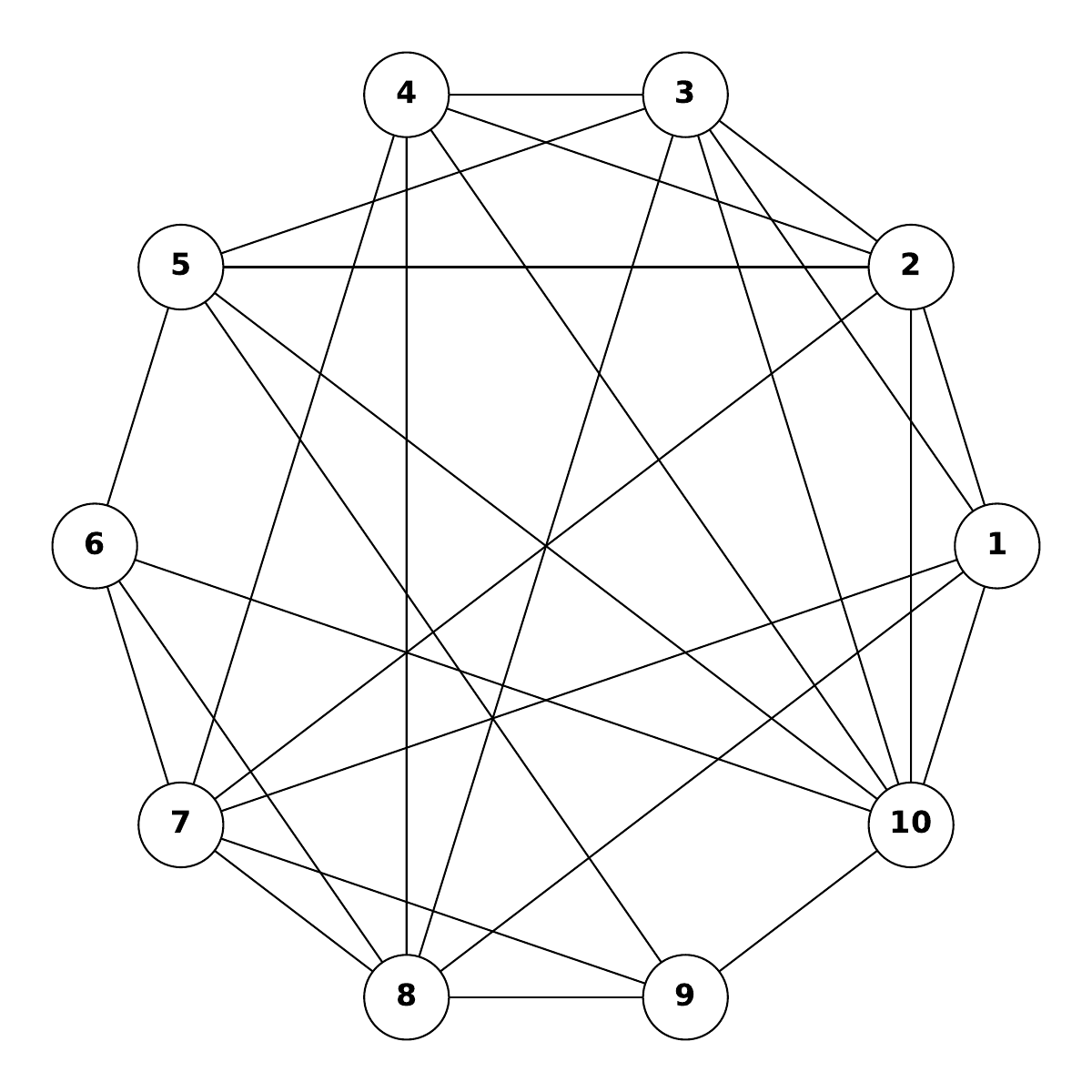}
    \end{subfigure}
    \hfill
        \begin{subfigure}{0.4\textwidth}
        \includegraphics[scale=0.3]{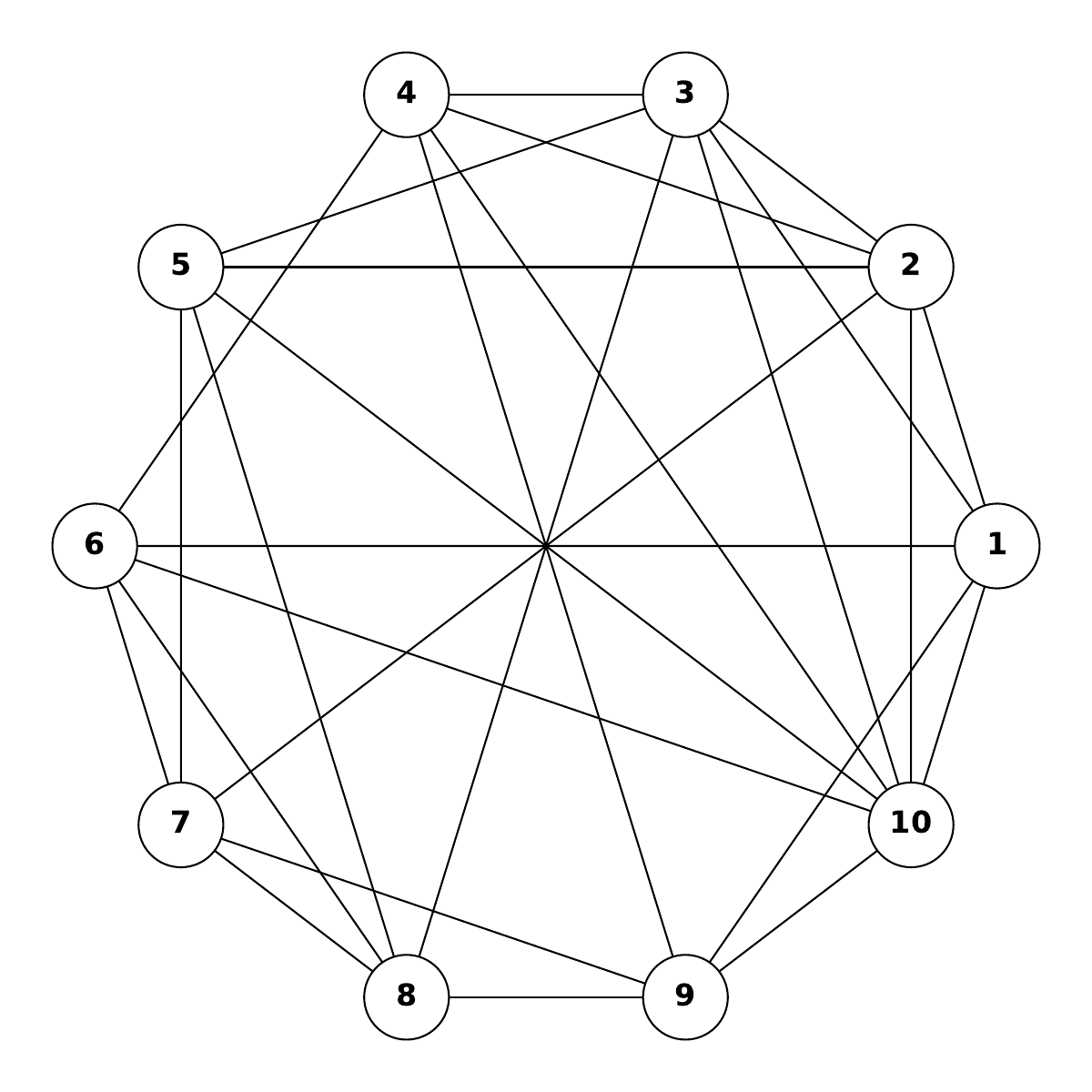}
    \end{subfigure}
    \hfill
        \begin{subfigure}{0.4\textwidth}
        \includegraphics[scale=0.3]{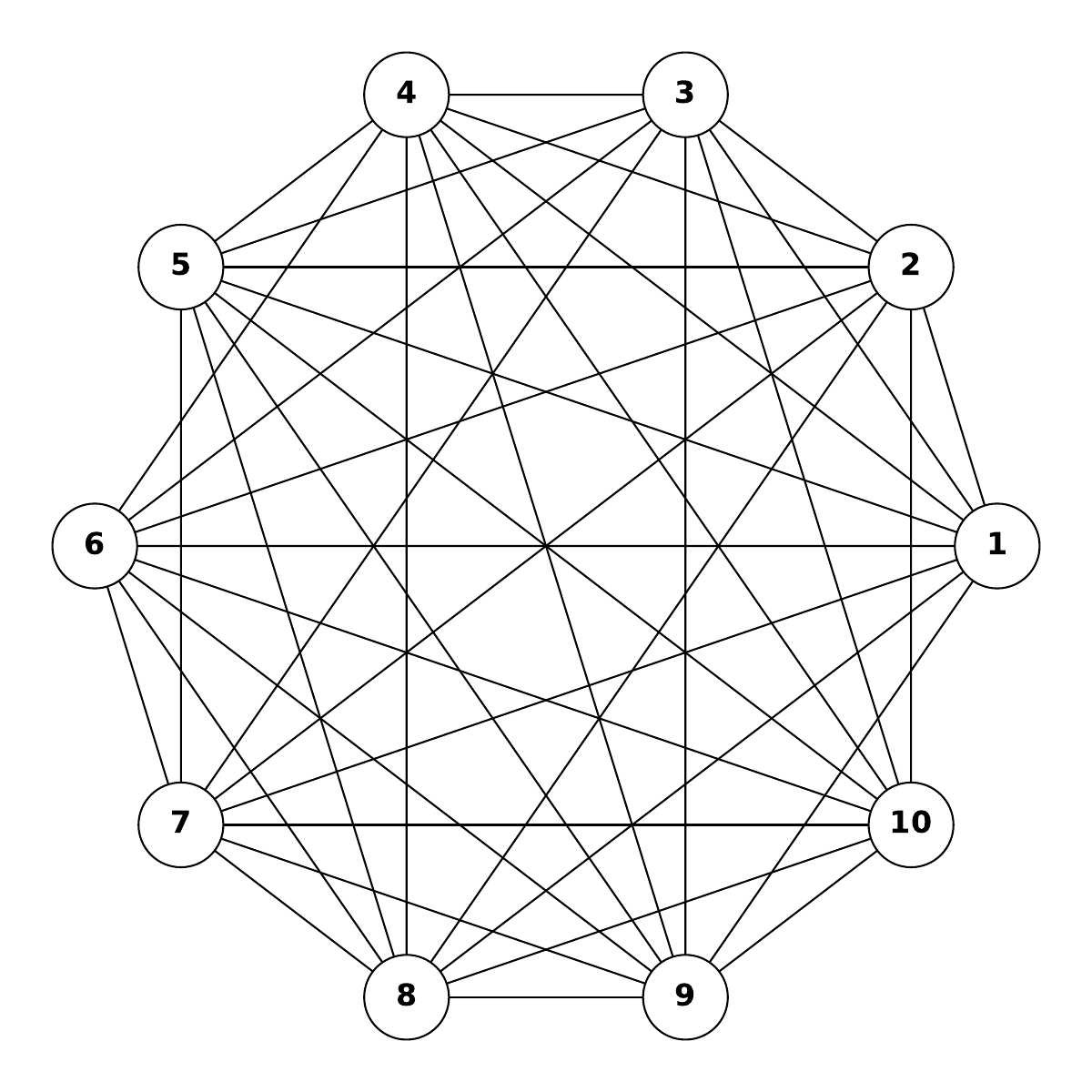}
    \end{subfigure}
\caption{Distance magic graphs of order $10$.}
\label{fig: dmg10}
\end{figure}

\subsection{Pseudo-code}
For a given value of $n$ and $k$, the Step $1$ of the algorithm involves the generation of all subsets of a set $\{1,2, \cdots, n\}$ with sum $k$. The algorithm may try all subsets of the given set in worst case. As a result, the time complexity of the above algorithm is exponential. However, our primary goal is to efficiently generate distance magic graphs only for specific values of $n$ and $k$, so we do not place a strong emphasis on the time complexity of this algorithm.

\smallskip
We provide the complete pseudo-code for the algorithm. 

\begin{algorithm}[]
\caption{Distance magic graph generator}
\begin{algorithmic}[1]
\State\textit{{\textbullet} To generate $k$-sum subsets}
    \Function{getSubset}{$A, k, n, \text{ans}, \text{List}$}
    \State $A \gets \text{list of integers from 1 to } n$
    \State $\text{ans} \gets \text{empty list}$
    \State $\text{List} \gets \text{empty list}$
        \If{$k = 0$}
            \State $ans1 \gets ans$
            \State List.append(ans1)
            \State \Return
        \EndIf

        \If{$k \neq 0$ \textbf{and} $n = \text{length of } A - 1$}
            \State \Return
        \EndIf

        \If{$k - A[n] \geq 0$}
            \State ans.append$(A[n])$
            \State getSubset$(A, k - A[n], n - 1, \text{ans}, \text{List})$ \Comment{include the element}
            \State ans.pop()
        \EndIf

        \State getSubset$(A, k, n - 1, ans, List)$ \Comment{do not include the element}
    \EndFunction
        \smallskip
\State \textit{{\textbullet} To generate neighborhood sets $NS$}
        \Function{generate\_neighbors}{$n, k$}
        \State \text{getSubset}(A, $k$, $n - 1$, \text{ans}, \text{List})
        \State $\text{neighborhoods} \gets \text{empty list}$
        \For{$i \gets 1$ to $n$}
            \State $N_i \gets \text{empty list}$
            \For{\textbf{each} subset \textbf{in} List}
                \If{$i \notin \text{subset}$}
                    \State $N_i \gets \text{append subset to } N_i$
                \EndIf
            \EndFor
            \State $\text{append } N_i \text{ to neighborhoods}$
        \EndFor
        \State \Return neighborhoods
    \EndFunction
\smallskip
        \State $\text{adj} \gets \text{list of empty lists of size } n+1$
        \State $\text{sum\_list} \gets \text{list of } n + 1 \text{ elements initialized to } 0$
        \State $\text{\text{cnt\_depth}} \gets 0$
        \State $\text{successful} \gets \text{list containing a single element } 0$
        \State $N \gets \text{generate\_neighbors}(n, k)$
        \State $\text{N.insert(0, [\;])}$
        \smallskip
        \Function{is\_symmetric}{$\text{adj, n, \text{depth}}$} \Comment{Checks the condition in (\ref{eq: check1}) in Step $3'$}
            \For{$i \gets n$ to $\text{depth} - 1$}
                \For{\textbf{each} neighbor \textbf{in} adj[i]}
                    \If{$\text{neighbor} \leq \text{depth}$}
                        \State \textbf{continue}
                    \EndIf
                    \If{$i \notin \text{adj[neighbor]}$}
                        \State \Return \textbf{False}
                    \EndIf
                \EndFor
            \EndFor
        \EndFunction
\end{algorithmic}
\end{algorithm}

\begin{algorithm}
\ContinuedFloat
\caption*{(Continued)}
\begin{algorithmic}[1]
        \Function{check}{$\text{adj\_list, sum\_list}$} \Comment{Checks the condition given in (\ref{eq: check2}) in Step $3'$} 
            \State $\text{done} \gets 1$
            \For{$i \gets 1 \text{ to length of adj\_list}$}
                \If{$\text{sum\_list}[i] \neq k$}
                    \State $\text{done} \gets 0$
                    \State \textbf{break}
                \EndIf
            \EndFor
        \EndFunction

        \smallskip
        \State\textit{{\textbullet} To construct a tree $T(n,k)$ and explore the successful branch}
        \Function{explorer}{$\text{N, adj, depth, breadth, sum\_list, successful}$} 
            \State $\text{adj}[\text{depth}] \gets \text{N}[\text{depth}][breadth]$ \Comment{Fixes first element of $NS(n)$ as a root of $T(n,k)$}
            \If{\text{depth} = 1}
                \For{$i \gets 1$ to $\text{length of adj}[\text{depth}]$}
                    \State $\text{sum\_list}[\text{adj}[\text{depth}][i]] \gets \text{sum\_list}[\text{adj}[\text{depth}][i]] + \text{depth}$
                \EndFor

                \If{\textbf{check}(adj, sum\_list) \textbf{and} \textbf{is\_symmetric}(adj, n, depth)}
                    \State $\text{successful}[0] \gets \text{successful}[0] + 1$
                    \For{$i \gets 1$ to $\text{length of adj}[\text{depth}]$}
                        \State $\text{sum\_list}[\text{adj}[\text{depth}][i]] \gets \text{sum\_list}[\text{adj}[\text{depth}][i]] - \text{depth}$
                    \EndFor
                    \State \text{Print the adjacency list of the graph}
                    \State \Return \textbf{True}
                \EndIf

                \For{$i \gets 1$ to $\text{length of adj}[\text{depth}]$}
                    \State $\text{sum\_list}[\text{adj}[\text{depth}][i]] \gets \text{sum\_list}[\text{adj}[\text{depth}][i]] - \text{\text{depth}}$
                \EndFor

                \State \Return \textbf{False}
            \EndIf

            \State $\text{flag} \gets \textbf{is\_symmetric}(\text{adj, n, depth})$

            \For{$i \gets 1$ to $\text{length of adj}[depth]$}
                \State $\text{sum\_list}[\text{adj}[\text{depth}][i]] \gets \text{sum\_list}[\text{adj}[\text{depth}][i]] + \text{depth}$
                \If{$\text{sum\_list}[\text{adj}[\text{depth}][i]] > k$}
                    \State $\text{flag} \gets 0$
                \EndIf
            \EndFor

            \If{$\text{flag} = 1$}
                \For{$i \gets 1$ to $\text{length of N}[\text{depth} - 1]$}
                    \If{\textbf{explorer}$(N, \text{adj}, \text{\text{depth}} - 1, i, \text{sum\_list}, \text{successful})$}
                        \State \Return \textbf{True}
                    \EndIf
                \EndFor
            \EndIf

            \For{$i \gets 1$ to $\text{length of adj}[\text{depth}]$}
                \State $\text{sum\_list}[\text{adj}[\text{depth}][i]] \gets \text{sum\_list}[\text{adj}[\text{depth}][i]] - \text{depth}$
            \EndFor

            \State $\text{adj}[\text{depth}] \gets \text{empty list}$
            \State \Return \textbf{False}
        \EndFunction

        \For{$i \gets 1$ to $\text{length of N}[n]$} \Comment{to construct $T(n,k)$ rooted at each element in $NS(n)$}
            \State \textbf{explorer}$(N, \text{adj}, n, i, \text{sum\_list}, \text{successful})$
        \EndFor

\end{algorithmic}
\end{algorithm}

\clearpage

\section{Conclusions}

With the results known earlier, and those proved in this paper, the magic constants are now completely characterized, as shown in the following theorem. 

\begin{theorem} \label{conclusion}
All positive integers except $1,2,4,6,8,12$, and $16$ are magic constants.
\end{theorem}

Not all magic constants arise from connected graphs. Also, for some integers, for example, $n = k = 7$, there is a unique distance magic graph $G(n,k)$. This leads naturally to the following questions. 
\begin{enumerate}
    \item Which magic constants are realised by connected distance magic graphs?
    \item For which pairs $(n, k)$, there is unique distance magic graph on $n$ vertices with magic constant $k$?
\end{enumerate}
\textbf{Acknowledgement:} The authors are grateful to Atharv Karandikar and Satyaprasad for their invaluable assistance in implementing the algorithm.

\bibliographystyle{plain}
\bibliography{refs}

\end{document}